\documentclass[a4paper,reqno,10pt]{amsart}

\usepackage[all]{xy}
\usepackage{amssymb}
\usepackage{graphicx}
\usepackage{amsmath,amsthm,amssymb,xypic,verbatim}

\newcommand{\p}{\mathbb{P}}

\newcommand{\N}{\mathbb{N}}

\newcommand{\T}{\mathbb{T}}

\newcommand{\oo}{\mathcal{O}}

\newcommand{\Sec}{\mathop{{\rm{\mathbb S}ec}}\nolimits}
\newcommand{\Span}[1]{\langle#1\rangle}

\theoremstyle{plain}

\newtheorem*{theorem}{Theorem}
\newtheorem{thm}{Theorem}
\newtheorem{pro}[thm]{Proposition}
\newtheorem{lem}[thm]{Lemma}   
\newtheorem{cor}[thm]{Corollary}
\newtheorem{claim}{Claim}
\theoremstyle{definition}

\newtheorem{notaz}[thm]{Notation}

\newtheorem{es}[thm]{Example}

\newtheorem{dfn}[thm]{Definition}
\newtheorem{rmk}[thm]{Remark}

\newcommand{\codim}[0]{\operatorname{codim}}

\begin{document}

\title[From non defectivity to identifiability]{From non defectivity to identifiability}
\date{\today}

\begin{abstract}
A projective variety $X\subset\p^N$ is $h$-identifiable if the generic
element in its $h$-secant variety uniquely determines $h$ points on
$X$. In this paper we propose an entirely new approach to study
identifiability, connecting it to the notion of secant defect.
In this way we are able to improve all known bounds on
identifiability. In particular we give optimal bounds for some  Segre and Segre-Veronese varieties and provide the
first identifiability statements for Grassmann varieties.
\end{abstract}

\author{Alex Casarotti}
\address{Dipartimento di Matematica e Informatica, Universit\`a di Ferrara, Via Machiavelli 35, 44121 Ferrara, Italy}
\email{csrlxa@unife.it}

\author{Massimiliano Mella}
\address{Dipartimento di Matematica e Informatica, Universit\`a di Ferrara, Via Machiavelli 35, 44121 Ferrara, Italy}
\email{mll@unife.it}
\maketitle
\section*{Introduction}
The notion of identifiability is ubiquitous both in applied and
classical algebraic geometry. In general we say that an
element $p$ of a projective space $\p^N$ is $h$-identifiable via a variety
$X$ if there is a unique way to write $p$ as linear combination of $h$
elements of $X$. In the classical setting this
very often translates into rationality problems and it is linked to the
existence of birational parameterizations. In the applied set up one
usually considers a tensor space and the identifiability allows to
reconstruct a tensor via a subset of special tensors defined by rank
conditions or other special requirements. For
applications ranging  from
biology to  Blind Signal Separation, data compression algorithms, and
analysis of mixture models, \cite{DDL1} \cite{DDL2} \cite{DDL3} \cite{KADL}
\cite{Si}, uniqueness of decompositions allows to solve the problem
once a solution is determined.
For all these reasons it is interesting and often crucial to
understand identifiability.

Over a decade ago the notion of $h$-weakly defective varieties has
been connected to identifiability of polynomials, \cite{Me}. This
provided the first systematic study of identifiability for Veronese
varieties. More
recently with the work of Luca Chiantini and Giorgio Ottaviani, \cite{CO1}, weakly
defective varieties have been substituted by $h$-tangentially weakly
defective varieties to study identifiability problems.
In both approaches to provide identifiability one has to check the
behavior of special linear systems and quite often this is done by an
ad hoc degeneration argument. As a consequence identifiability has been
proved in very few cases  and quite often the result obtained are not expected to be
sharp, \cite{CO1} \cite{BDdG} \cite{BC} \cite{BCO} \cite{Kr}.

In this paper we want to develop an
entirely new approach to study generic identifiability, see Definition~\ref{def:identifiability} for the precise set
up.  Starting from the seminal paper \cite{CC10}, where the geometry
of contact loci has been carefully studied, and the improvement
presented in \cite{BBC}, we  derive identifiability statements for non secant
defective varieties. Even if new this is not really surprising since  weakly defectiveness and
tangentially weakly defectiveness, thanks to Terracini Lemma, have
secant defectiveness as  a common ancestor.
With this new approach we are
able to translate all the literature on defective varieties into
identifiability statements, providing in many cases sharp
classification of $h$-identifiability. 
One of the results we prove in this direction is the following conditional
relation between identifiability and defectivity, we refer to
section~\ref{sec:notation} for the necessary definitions.

\begin{theorem}
  \label{thm:intro} Let $X\subset\p^N$ be an irreducible reduced
  variety. Assume that $h>\dim X$, $X$ is not $(h-1)$-tangentially
  weakly defective and it is not $h$-identifiable. Then $X$ is $(h+1)$-defective.
\end{theorem}

As already mentioned identifiability issues are particularly
interesting for tensor spaces. As a corollary we get the best asymptotic identifiability result so far for Segre, Segre-Veronese, and Grassmann varieties, that is, tensors and structured
tensors, see
section~\ref{sec:application}.  
As a sample  we  state the application to binary tensors.
\begin{theorem}
  \label{thm:segrep1} The Segre embedding of $n$ copies of $\p^1$,
  with $n\geq 5$ is
  $h$-identifiable for any $h\leq \lfloor\frac{2^n}{n+1}\rfloor-1$.  
\end{theorem}

Recall that the generic rank of the Segre embedding of $(\p^1)^n$
is $ \lceil\frac{2^n}{n+1}\rceil$, therefore our result shows generic
identifiability of all sub-generic binary tensors, qbits if you like
the quantum computing dictionary, in the perfect case, that is when
$\frac{2^n}{n+1}$ is an integer, and all but the
last one in general, as predicted by the conjecture posed in \cite{BC}.

The starting point of our analysis was the observation that, in all
known examples, when a variety $X$ is not $h$-identifiable then any
element in $\Sec_{h+1}(X)$ has infinitely many decompositions, \cite{CMO}
\cite{BBC} \cite{BCO} \cite{BV} \cite{COV1}. Going back to the ideas in \cite{Me} we
realized that  the  best way to
use this observation is to set a connection between the abstract
secant map and the tangential projection. Via this we
reduce the problem to study fiber type tangential projections.
The latter is accomplished via the construction of a map from the Hilbert scheme of points of the contact loci of
$h$-tangentially weakly defective varieties to a suitable Grassmannian. Under
the right assumptions this is proved to be of fiber type and it allows us to
connect defectivity and non identifiability. This, together with an improvement of the
contact loci geometry studied in \cite{CC10} and \cite{BBC}, lead us to derive
identifiability from non defectivity under very mild hypothesis.

\begin{theorem}
  \label{thm:twd_defect} Let $X\subset\p^N$ be a  smooth variety.
  Assume that  $\pi^X_k:\Sec_k(X)\to\p^N$ is generically finite and 
  $k>2\dim X$. 
  Then $X$
  is $(k-1)$-identifiable.
\end{theorem}

The paper is structured as follows. In section~\ref{sec:main} we study
the geometry of contact locus and prove the main general results about
the connections between defectivity and non identifiability.
In the final section we apply our techniques to varieties that are
meaningful for tensor decomposition and streamline our argument for those.

We are much indebted with Luca Chiantini for many conversations on the
subject and for explaining us the connection between tangentially
weakly defective varieties and identifiability when we started to work
on the subject.

\section{Notation}\label{sec:notation}

We work over the complex field. A projective variety $X\subset\p^N$ is
non degenerate if it is not contained in any hyperplane.

Let $X\subset\p^N$ be an irreducible and reduced non degenerate
variety. Let $X^{(h)}$ be the $h$-th symmetric product of $X$,
that is the variety parameterizing unordered sets of $h$ points of $X$. 
Let $U_h^X\subset X^{(h)}$ be the smooth locus, given by sets of $h$
distinct points.
\begin{dfn}
  A point $z\in U^X_h$ represents a
set of $h$ distinct point, say $\{z_1,\ldots, z_h\}$. We say that a point $p\in \p^N$ is in the span of
 $z$, $p\in\langle z\rangle$,  if it is a linear combination of the
 $z_i$.
\end{dfn}

With this in mind we define


\begin{dfn} The {\it
    abstract $h$-Secant variety} is the irreducible and reduced variety
$$\textit{sec}_{h}(X):=\overline{\{(z,p)\in U_h^X\times\p^N| p\in z\}}\subset
X^{(h)}\times\p^N.$$ 

Let $\pi:X^{(h)}\times\p^N\to\p^N$ be the projection onto the second factor. 
The {\it $h$-Secant variety} is
$$\Sec_{h}(X):=\pi(sec_{h}(X))\subset\p^N,$$
and $\pi_h^X:=\pi_{|sec_{h}(X)}:sec_{h}(X)\to\p^N$ is the $h$-secant map of $X$.

The irreducible variety $\textit{sec}_{h}(X)$ has dimension $(hn+h-1)$.
 One says that $X$ is
\textit{$h$-defective}
if $$\dim\Sec_{h}(X)<\min\{\dim\textit{sec}_{h}(X),N\}.$$
\end{dfn}

\begin{rmk} If $X$ is $h$-defective then the $h$-secant map is of
  fiber type.
  
Note that a general point in $\sec_h(X)$  is a linear combination of
$h$ points of $X$. Thanks to the non degeneracy assumption  a general
point in $\sec_h(X)\subsetneq\p^N$  is not a linear combination of less points on $X$.

The tricky part in studying secant varieties is the closure. Many different things can happen, the $h$ points can
group in infinitely near cluster or positive dimensional intersection
can appear. As a matter of facts these special loci are really
difficult to control and the main advantage to use birational geometry
is the opportunity to get rid of them. 
\end{rmk}

\begin{dfn}
  Let $X\subset\p^N$ be a non degenerate subvariety. We say that a
  point $p\in\p^N$ has rank $h$ with respect to $X$ if $p\in \langle z\rangle$, for
  some $z\in U_h^X$ and $p\not\in \langle z^\prime\rangle$ for any $z^\prime\in
  U_{h^\prime}^X$, with $h^\prime<h$.
\end{dfn}

\begin{rmk}
  With this in mind it is easy to produce examples
of limits of rank $h$ points with different rank. If we let one of the
point degenerate to the span of the others we lower the rank. If we
let two points collapse the rank, generically, increases.
\end{rmk}

\begin{dfn}\label{def:identifiability} A point $p\in\p^N$
  is $h$-identifiable with respect to $X\subset\p^N$ if $p$ is of rank
  $h$ and $(\pi^X_h)^{-1}(p)$ is a single point. The variety $X$ is
  said to be $h$-identifiable if $\pi_h^X$ is a birational map, that
  is the general point of $\Sec_h(X)$ is $h$-identifiable
\end{dfn}

It is clear, by the above remark, that when $X$ is $h$-defective, or
more generally when $\pi_h^X$ is of fiber type $X$ is not $h$-identifiable.

The next ingredient we need to introduce is Terracini Lemma.
\begin{thm}[Terracini Lemma]\cite{CC1}
  \label{th:Terracini} Let $X\subset\p^N$ be an irreducible
  variety. Then we have
  \begin{itemize}
 \item for any $x_1,\ldots,x_k\in X$ and $z\in\Span{x_1,\ldots,x_k}$
$$\Span{\T_{x_1}X,\ldots,\T_{x_k}X}\subseteq \T_z\Sec_k(X), $$
 \item there is a dense open set $U\subset X^{(k)}$ such that
$$\Span{\T_{x_1}X,\ldots,\T_{x_k}X}=\T_z\Sec_k(X), $$
for a general point $z\in \Span{x_1,\ldots,x_k}$ with $(x_1,\ldots,x_k)\in U$.
  \end{itemize}
\end{thm}

Terracini Lemma yields a direct consequence of $h$-defectiveness. If
$X$ is $h$-defective then the general fiber of $\pi_h^X$ has positive
dimension. Therefore by Terracini the general hyperplane tangent at
$h$ points of $X$ is singular along a positive dimensional subvariety.
This
property does not characterize defective varieties.
\begin{dfn}
  \label{def:weaklydef} Let $X\subset\p^N$ be a non degenerate
  variety. The variety $X$ is said $h$-weakly defective if the general
  hyperplane singular along $h$ general points is singular along a
  positive dimensional subvariety. 
\end{dfn}

There is a direct connection, proven in \cite{CC1}, between
$h$-weakly defectiveness and identifiability.
\begin{thm}
  \label{thm:nonwd_implies_ident} If $X$ is not $h$-weakly defective
  then it is $h$-identifiable.
\end{thm}

The main problem is that it is quite hard in general to verify if a
variety is $h$-weakly defective.

  To overcome this problem the notion of tangentially weakly defective
  varieties has been introduced, \cite{CO1}. Here we follow the
  notations of \cite{BBC}.

  For a subset $A=\{x_1,\ldots,x_h\}\subset X$ of general points we set
  $$ M_A := \langle\bigcup_{i}\T_{x_i}X\rangle.$$
By Terracini Lemma  the space $M_A$ is the tangent space to $\Sec_h(X)$
at a general point in $\langle A\rangle$.
\begin{dfn}
  The tangential $h$-contact locus $\Gamma_h:=:\Gamma(A)$ is the closure in $X$ of the union of all the irreducible components which contain at least one point of $A$, of the locus of points of $X$ where $M_A$ is tangent to $X$.
We will write $\gamma_h := \dim \Gamma(A)$.
We say that $X$ is $h$-twd (tangentially weakly defective) if $\gamma_h > 0$.
\end{dfn}

  \begin{rmk}
    It is clear that if $X$ is $h$-twd then it is $h$ weakly
    defective, it is $(h+1)$-twd and $\Gamma_h\subseteq \Gamma_{h+1}$. Using scrolls it is not too difficult to produce explicit
    examples of varieties that are $h$-weakly defective but are not
    $h$-twd, see also Remark~\ref{rmk:optimal_c}.
  \end{rmk}

For what follows it is useful to introduce also the notion of
tangential projection.
\begin{dfn}
  Let $X\subset\p^N$ be a variety and $A=\{x_1,\ldots,x_h\}\subset X$ a
  set of general points. The $h$-tangential projection (from A) of $X$ is
  $$\tau_h:X\dasharrow\p^M$$
  the linear projection from $M_A$. That is, by Terracini Lemma,   the projection from the
  tangent space of a
  general point $z\in \Span{A}$ of $\Sec_h(X)$ restricted to $X$.
\end{dfn}

\section{Relation between twd and defectivity}\label{sec:main}
We start  collecting  properties of the tangential contact loci that will be
useful for our purpose.
\begin{thm}
  \label{thm:proprieta_contact_locus} Let $X\subset\p^N$ be an
  irreducible, reduced, and non degenerate variety. Let $A\subset X$ be
  a set of $h$ general points and $\Gamma$ the associated contact
  locus. Assume that $\Sec_{h-1}(X)\subsetneq\p^N$.
  Then we have:
  \begin{itemize}
  \item[a)] $\Gamma$ is equidimensional and it is either irreducible
    (type I) or reduced (type II) with exactly $h$ irreducible component, each of them
    containing a single point of $A$ \cite[Proposition 3.9]{CC10},
    \item[b)] $\Span{\Gamma}=\Sec_h(\Gamma)$ and $\Sec_i(\Gamma)\neq
      \Span{\Gamma}$ for $i<h$ \cite[Proposition 3.9]{CC10},
        \item[c)] for $z\in\Span{\Gamma}$ general
          $\pi^X_h((\pi^X_h)^{-1}(z))\subset\Span{\Gamma}$,
          \cite[Proposition 3.9]{CC10},
            \item[d)] if we are in type I $\gamma_h>\gamma_{h-1}$,
              \cite[Lemma 3.5]{BBC}
                \item[e)] if  $\gamma_h=\gamma_{h+1}$  and $\Sec_{h+1}(X)$ is
                  not defective and does not fill up $\p^N$ we are in
                  type II,  the irreducible components of both contact
                  loci are linearly
                  independent linear spaces, \cite[Lemma
                  3.5]{BBC},
                  \item[f)] if we are in type I and   $\Sec_{h+1}(X)$ is
                  not defective and does not fill up $\p^N$ then
                  $\Gamma_{h+1}$ is of type I.
  \end{itemize}
\end{thm}
\begin{proof} Points a)--e) are proved in the cited papers under the
  assumption that $\Sec_h(X)\subsetneq\p^N$. Points a)--d)  are immediate when
  $\Sec_h(X)=\p^N$ and $\Sec_{h-1}(X)\subsetneq\p^N$.
  
  We have only to prove point f). Let $A=\{x_1,\ldots,x_h\}$ and
  $B=A\cup\{x_{h+1}\}$ be general sets in $X$. Assume that $\Gamma(B)$
  is of type II. By definition $\Gamma(A)\subset\Gamma(B)$, on the
  other hand  by point a) the irreducible component of  $\Gamma(B)$
  through $x_1$ does not contain $x_2$ and therefore it cannot contain
  $\Gamma(A)$. This contradiction proves the claim. 
\end{proof}

 From the point of view of  identifiability the notions of weakly
 defectiveness  and twd  behave the same.
The following proposition is well known to the experts in the
field but we were not able to found a written version of it.
\begin{pro}\cite{Cpc} Let $X\subset\p^N$ be an irreducible, reduced, and
  non degenerate variety. Assume that $X$ is not $h$-twd, then $X$ is $h$-identifiable.
  \label{pro:non_twd_implica_ident} 
\end{pro}
\begin{proof} 
  Assume that $X$ is not $h$-identifiable and let $z\in\Sec_h(X)$ be
  a general point. Let $z\in \Span{x_1,\ldots x_h}$, for $x_i$ general
in $X$. The existence of a different decomposition yields a new set
  $\{y_1,\ldots y_h\}\subset X$ such that $z\in\Span{y_1,\ldots,
    y_h}$. Moving the point $z$ in the linear space $\Span{x_1,\ldots,
    x_h}$ yields a positive dimensional contact locus. 
\end{proof}

\begin{rmk} We want to stress that $h$-identifiability is not
  equivalent to non $h$-twd.
 In \cite{COV2} and \cite{BV} are described examples of Segre and
 Grassmannian varieties that
 are $h$-identifiable but $h$-twd. 
\end{rmk}

We aim to study the relation between twd and defectivity. The next
lemma is a first step in this direction.

\begin{lem} Let $X\subset\p^N$ be an irreducible, reduced, and non degenerate variety of
  dimension $n$,
  $$\pi_k^X:\sec_k(X)\to\p^N$$
  the $k$-secant  map,
  $\tau^X_{k-1}:X\dasharrow \p^M$ the $(k-1)$-tangential projection, and
  $\Gamma:=\Gamma(x_1,\ldots,x_k)$ the $k$-contact locus associated to the
  genral points  $x_1,\ldots, x_k$.
  \begin{itemize}
  \item[i)] The map $\pi^X_k$ is of fiber type if and only if $\tau^X_{k-1}$ is of
    fiber type.
  \item[ii)] Let $\{x_1,\ldots,x_k,y_1,y_2\}$ be general points. Then
    $$\dim (\Gamma(x_1,\ldots,x_k, y_1)\cap
    \Gamma(x_1,\ldots,x_k,y_2))>0$$
    in a neighborhood of $x_i$ only if either $X$ is $k$-twd or
$\pi^X_{k+2}$ has positive dimensional fibers.
\item[iii)] The map $(\tau^X_{k-1})_{|\Gamma}:\Gamma\dasharrow\p^{\gamma_k}$ is
  either of fiber type or dominant.
  \end{itemize}
  \label{lem;easy_rem}
\end{lem}
\begin{proof} i) By Terracini Lemma $\pi_k^X$ is
  of fiber type if and only if
  $$\T_z\Sec_{k-1}(X)\cap
    \T_yX\neq\emptyset$$ for $y\in X$ general. This condition is clearly
    equivalent to have $\tau^X_{k-1}$  of fiber type.

ii) Assume that $X$ is not $k$-twd and  $\dim (\Gamma(x_1,\ldots,x_k, y_1)\cap
\Gamma(x_1,\ldots,x_k,y_2))>0$ in a neighborhood of $x_i$. Set
$$M_{A_i}=\langle \T_{x_1}X,\ldots,\T_{x_k}X,\T_{y_i}X\rangle,$$
the variety $X$ is not $k$-twd therefore 
$$M_{A_1}\cap M_{A_2}\supsetneq \langle \T_{x_1}X,\ldots,\T_{x_k}X\rangle. $$
In particular we have
$$(M_{A_1}\cap M_{A_2})\cap \T_{y_i}X\neq\emptyset,$$
and hence
$$\langle \T_{x_1}X,\ldots,\T_{x_k}X,\T_{y_1}X\rangle\cap \T_{y_2}X\neq\emptyset.$$
This shows, by the generality of the points and point i) 
that $\pi^X_{k+2}$ is of fiber type.

iii) Assume that $(\tau_{k-1}^X)_{|\Gamma}$ is not of fiber type. Then
by point b) of Theorem~\ref{thm:proprieta_contact_locus} we have $\dim
\Span{\Gamma}=k(\gamma_k+1)-1$. Hence
$(\tau_{k-1}^X)_{|\Gamma}=\tau_{k-1}^\Gamma$ and both maps are
dominant onto $\p^{\gamma_k}$.
\end{proof}

Next we prove a general statement for type II contact loci.
\begin{lem}\label{lem:clII_ok}  Let $X\subset\p^N$ be an irreducible,
  reduced, and non degenerate variety. Assume that:
  \begin{itemize}
  \item[a)]   $X$ is $k$-twd,
  \item[b)] $X$ is not $(k-1)$-twd
    \item[c)] the $k$-contact locus is of type II.
  \end{itemize} 
 Then $\pi_{k+1}^X$ is of fiber type.
\end{lem}
\begin{proof} By point i) in Lemma \ref{lem;easy_rem} it is enough to
  prove that $\tau_k^X$ is of fiber type. Then by projection it is
  enough to prove the latter  for $k=2$.
  Let $\{x_1,x_2,y\}\subset X$ be a set of general points and
  $\Gamma=\Gamma(x_1,x_2,y)$ the contact locus associated
  to
  $\{x_1,x_2,y\}$. To conclude the proof it is enough to
  prove that $\Span{\T_{x_1},\T_{x_2}}\cap\T_xX\neq\emptyset$,
  for $x\in \Gamma$ a general point.
  
  For a general point $p\in \Gamma$ we set
  $$\Gamma^i_p\subset\Gamma(x_i,p)$$ the irreducible component of the
  contact locus $\Gamma(x_i,p)$ through
  $p$.  The contact locus is of type II, therefore $\Gamma^i_p\not\ni
  x_1, x_2$.
  Note that for a general point $x\in \Gamma^1_p$ we have
  $\T_xX\subset\Span{\T_{x_1}X,\T_pX}$. Then by
  semicontinuity for any point $w\in \Gamma^1_p$ there is a linear space of
  dimension  $n$, say $A_w\subseteq\T_wX$, contained in the span.
  
Set
  $$\T(\Gamma^1_p)=\Span{A_w}_{w\in\Gamma^1_p}.$$
  We may assume that $X$ is not $2$-defective, otherwise there is
  nothing to prove,  that is
  \begin{equation}
    \label{eq:2-def}
  \T_{x_1}X\cap\T_{x_2}X=\emptyset, 
  \end{equation}
  and, since $y$ is general,
  \begin{equation}
    \label{eq:n+1usual}
     \codim_{\T(\Gamma^1_y)}(\T(\Gamma^1_y)\cap\T_{x_1}X)=n+1.
  \end{equation}
  The variety $X$ is not $1$-twd, then there are points $z\in\Gamma^1_y$ with $A_z\cap\T_{x_1}X\neq\emptyset$.
  Let $z\in \Gamma^1_y$ be a point with
  \begin{equation}
    \label{eq:zandy}
     A_z\cap\T_{x_1}X\neq\emptyset,
   \end{equation}
   The contact locus is of type II, therefore $z\neq x_1$. We want to
   stress that this is  the only point in the  proof where we use the
   assumption that $\Gamma$ is of type II.

   If $A_z\cap\T_{x_2}X\neq \emptyset$, by Equations~(\ref{eq:2-def})~and~(\ref{eq:n+1usual}) we have
     $$\codim_{\T(\Gamma^1_y)}(\T(\Gamma^1_y)\cap\Span{\T_{x_1},\T_{x_2}})\leq
     n$$
     and we conclude
    $\T_yX\cap\Span{\T_{x_1},\T_{x_2}}\neq\emptyset$, that is  $\tau_2^X$ is of fiber type.

 Assume that  $A_z\cap\T_{x_2}X=\emptyset$. Then we
    consider the span $\Span{A_z,\T_{x_2}}$. By semicontinuity to this linear space is
    associated a contact locus  and we set  $\Gamma_z^2$ its
    irreducible component passing
through $z$. As before we have
$$\codim_{\T(\Gamma_z^2)}(\T(\Gamma_z^2)\cap\T_{x_2})= n+1,$$
and by Equations~(\ref{eq:2-def})~and~(\ref{eq:zandy}) we conclude that
$$\codim_{\T(\Gamma_z^2)}(\T(\Gamma_z^2)\cap\Span{\T_{x_1},\T_{x_2}})\leq
n.$$

This yields
\begin{equation}
  \label{eq:span_x_1x_2}
   A_w\cap  \Span{\T_{x_1},\T_{x_{2}}}\neq \emptyset,
\end{equation}
  for any point $w\in\Gamma_z^2$. We have $z\neq y_1$, then  the general choice of the points
  $x_i$,  and the assumption that $X$ is not $2$-defective ensure that 
  \begin{equation}
    \label{eq:no_x_1}
      A_w\cap\T_{x_1}X=\emptyset
    \end{equation}
 for general
 $w\in\Gamma_z^{2}$.


 We set $\Gamma_w^1$ the irreducible
 component through $w$ of the contact locus associated to $\Span{A_w,\T_{x_1}X}$.
 Again $z\neq x_1$ and the general choice of the $x_i$  ensure that 
 $z\not\in\Gamma_w^1$. In particular
 $$\Gamma_w^1\neq\Gamma_z^2.$$
Set
      $$S^2:=\bigcup_{v\in\Gamma^1_y
          \mbox{general}}\Gamma_v^2.$$
      Then $\Gamma_z^2$ is in the closure of $S^2$ and, for
      $p\in\Gamma_y^1$ general, $\Gamma_y^1$ is in the closure of
      $$\bigcup_{w\in\Gamma^2_p
        \mbox{general}}\Gamma_w^1.$$
      Hence $\Gamma_y^1$ is in the closure of
      $$S^1:=\bigcup_{w\in\Gamma^2_z
        \mbox{general}}\Gamma_w^1. $$
   In particular the general point
      of $S^1$ is a general point of $X$.  By construction we have
 $$\codim_{\T(\Gamma_w^1)}(
 \T(\Gamma_w^1)\cap\T_{x_1}X)\leq n+1.$$
 Equations~(\ref{eq:span_x_1x_2}) and~(\ref{eq:no_x_1}) then give
$$\codim_{\T(\Gamma_w^1)}(\T(\Gamma_w^1)\cap\Span{\T_{x_1},\T_{x_2})}\leq
  n$$
  and this concludes the proof.
  \end{proof}

We are ready to prove our main result  that connects  twd and defectivity.

\begin{thm} Let $X\subset\p^N$ be an irreducible, reduced, and non degenerate variety of
  dimension $n$. Assume that:
  \begin{itemize}
  \item[a)]   $X$ is $k$-twd,
  \item[b)] $X$ is not $(k-1)$-twd
    \item[c)] $k> n$ and
  $N\geq(k+1)(n+1)-1$.
  \end{itemize} 
 Then $\pi_{k+1}^X$ is of fiber type.
\label{th:ktwd}
\end{thm}

\begin{proof} Thanks to Lemma~\ref{lem:clII_ok} we may assume that the
  contact locus is of type I.
By hypothesis the variety $X$ is $k$-twd. Let $A=\{x_1,\ldots, x_k\}\subset X$ be a
set of general points and $\Gamma:=\Gamma(A)$ the associated contact
locus of dimension $\gamma>0$.
Let $z\in\Span{A}$ be a general point,
$\tau_k:=\tau^X_k:X\dasharrow\p^M$ the associated 
$k$-tangential projection, and $y\in X$ a general
point. For a general set $Y:=\{y_1,\ldots,
y_{k-1}\}\subset\Gamma$ let $\Gamma(Y\cup\{y\})$ be the contact locus associated
to $\{y_1,\ldots,y_{k},y\}$.

 Assume that $\pi_{k+1}^X$ is not of fiber type. 
then, by i) in Lemma~\ref{lem;easy_rem} $\tau_{k}$ is not of fiber
type and by point iii) in Lemma~\ref{lem;easy_rem},
$\tau_k(\Gamma(Y\cup\{y\})$ is a linear space of dimension
$\gamma$ through $z:=\tau_k(y)$. This gives a map
$$\chi:Hilb_{k-1}(\Gamma)\dasharrow {\mathbb G}(\gamma-1, M-1). $$

The point $z$ is  smooth hence all these linear spaces sit  in
$\T_z\tau_k(X)\cong\p^n$.  In other words we have a map
$$\chi:Hilb_{k-1}(\Gamma)\dasharrow {\mathbb G}(\gamma-1, n-1)\subset{\mathbb G}(\gamma-1, M-1). $$
Note that $\dim {\mathbb G}(\gamma-1, n-1)=\gamma(n-\gamma)$
and $\dim Hilb_{k-1}(\Gamma)=(k-1)\gamma$. By hypothesis $k> n$
and $\gamma>0$
hence we have 
$$(k-1)\gamma> \gamma(n-\gamma).$$
Then the map $\chi$ is of fiber
type and fibers have, at least,  dimension $\gamma(k-n+\gamma-1)$.

Set $[Y_1], [Y_2]\in\chi^{-1}([\Lambda])$ general points, for $[\Lambda]\in
\chi(Hilb_{k-1}(\Gamma))\subset{\mathbb G}(\gamma-1, n-1)$ a general point. The variety $X$ is not
$(k-1)$-twd and we are assuming that $\pi_{k+1}^X$ is not of fiber type
therefore, by ii) in Lemma~\ref{lem;easy_rem},
$$\dim(\Gamma\cap\Gamma(Y_i\cup\{y\}))=0,$$
in a neighborhood of $y_i$. Since the fiber of $\chi$ is positive
dimensional we have
\begin{equation}
  \label{eq:diff_irr}
  \Gamma(Y_1\cup\{y\})\not\supset Y_2.
\end{equation}
The
contact loci are irreducible then, by Equation~(\ref{eq:diff_irr}), we
conclude that 
\begin{equation*}
  \Gamma(Y_1\cup\{y\})\neq
\Gamma(Y_2\cup\{y\}).
\end{equation*}
Therefore, by point iii) in Lemma~\ref{lem;easy_rem}, the positive dimensional fiber of $\chi$ induces a positive dimensional
fiber of $\tau_{k}$ and we derive,   by point i)
Lemma~\ref{lem;easy_rem}, the contradiction that
that $\pi_{k+1}^X$ is of fiber type.
\end{proof}

\begin{rmk}\label{rmk:optimal_c} Both assumption b) and c) alone are
  reasonable and not over-demanding. Unfortunately the combination of
  them is quite restrictive and narrows the range of application we
  are aiming at.
  
  We believe the statement is not optimal with respect to assumption
  c). But we are not sure it is true, in full generality, without any assumption of this
  kind. On the other hand we strongly believe that for many
  interesting varieties, like Segre, Grassmannian, Veronese and their combinations,  
  twd can occur only one step before the secant map becomes of fiber
  type.
  This is not the case for weakly defectiveness as is shown in
  \cite{BV}. In \cite[Theorem 1.1 a)]{BV} it is proven that
  ${\mathbb G}(2,7)$ is 2 and 3 weakly defective without being
  3-defective. Note that this variety is  3-twd and it is not 2-twd.
\end{rmk}



  The next result generalizes the main result in \cite{BBC} and it allows to avoid the bottleneck introduced by conditions b) and
  c)  of Theorem~\ref{th:ktwd}  in  many interesting situations.

  \begin{lem}
  \label{lem:type2} Let $X\subset\p^N$ be an irreducible, reduced, and
  non degenerate variety.
  Assume that $X$ is not $1$-twd and $\pi_{k+1}^X$ is generically
  finite, in particular $X$ is not $(k+1)$-defective. If 
  $X$ is $k$-twd  then  $\gamma_k<\gamma_{k+1}$.
\end{lem}

\begin{proof}The variety $X$ is  not $1$-twd, then we may assume, without loss of generality,
  that
  $$\gamma_{k-1}<\gamma_k=\gamma_{k+1}.$$
  
  Then $\gamma_{k+1}<n$ and $\Sec_{k+1}(X)\subsetneq\p^N$, hence  
by e) in Theorem~\ref{thm:proprieta_contact_locus}, the contact loci
are of type II and linearly independent linear spaces.
Fix $\{x_1,\ldots,x_{k},y\}\subset X$ a set of general points and
let
$$\Gamma(x_1,\ldots,x_{k},y)=\cup_1^k P_i\cup P_y$$
the contact locus.
Moreover the assumption  $\gamma_k=\gamma_{k+1}$ and point a) in
Theorem~\ref{thm:proprieta_contact_locus}   force
$$\Gamma(x_1,\ldots,x_{k-1},y)=\cup_1^{k-1} P_i\cup P_y,$$
with the same $P_i$'s.
Then
$$\bigcap_{y\in X} \Span{\T_{x_1}X,\ldots,\T_{x_{k-1}}X,\T_yX}\supset
\Span{\T_zX}_{z\in P_i, \mbox{i=1,\ldots,
    k-1}}$$
We are assuming that $\gamma_{k-1}<\gamma_k$ therefore
$$P_i\not\subset\Gamma(x_1,\ldots, x_{k-1}),$$
and we have a proper inclusion
$$ \Span{\T_zX}_{z\in P_i, \mbox{i=1,\ldots,
    k-1}}\supsetneq\langle \T_{x_1}X,\ldots,\T_{x_{k-1}}X\rangle. $$

Set
$$M_{A_i}=\langle \T_{x_1}X,\ldots,\T_{x_{k-1}}X,\T_{y_i}X\rangle,$$
for general points $y_1, y_2\in X$.
Then we have
$$M_{A_1}\cap M_{A_2}\supset\Span{\T_zX}_{z\in P_i, \mbox{i=1,\ldots,
    k-1}}\supsetneq\langle \T_{x_1}X,\ldots,\T_{x_{k-1}}X\rangle. $$
and we conclude that
$$(M_{A_1}\cap M_{A_2})\cap \T_{y_i}\neq\emptyset.$$ 
This shows that
$$\langle \T_{x_1}X,\ldots,\T_{x_{k-1}}X,\T_{y_1}X\rangle\cap \T_{y_2}X\neq\emptyset.$$
hence the $k$-tangential projection $\tau^X_k$ is of fiber type
and by Lemma~\ref{lem;easy_rem} we derive the contradiction that
$\pi_{k+1}^X$ is of fiber type.
\end{proof}
\begin{rmk}
  Let us recall that $1$-twd varieties are classified in \cite{GH} and
  are essentially generalized developable varieties. In particular
  they are ruled by linear spaces and,
  with the unique exception of linear spaces,  they are singular.
\end{rmk}

We are ready to apply the above results to get not tangentially weakly
defectiveness  and hence identifiability  statements.

\begin{cor}
  \label{cor:twd_defect} Let $X\subset\p^N$ be an irreducible,
  reduced, and non degenerate variety that is not $1$-twd, for instance a smooth variety
  or a variety that is not covered by linear spaces.
  Assume that  $\pi^X_k$ is generically finite and 
  $k\geq \dim X$.

  Then $X$
  is not $(k-\dim X)$-twd and  it is not $(k-\dim X+1)$-twd if
  $\pi^X_k$ is not dominant.
If moreover either $k> 2\dim X$ or $\pi_k^X$ is not dominant and $k\geq
2\dim X$ then $X$ is not
$(k-1)$-twd.

In all the above cases $X$ is $h$-identifiable.
\end{cor}

\begin{proof} By hypothesis $\pi_h$ is generically finite for any
  $h\leq k$. Then by Theorem~\ref{lem:type2} if it is $j$-twd
  $$\gamma_j<\gamma_{j+1}.$$ The contact locus is a subvariety of $X$,
  hence $\gamma_{k-\dim X}=0$. This proves the first statement.

  If $\pi_k^X$ is not dominant then the contact locus is a proper
  subvariety and we have $\gamma_{k-\dim X+1}=0$.

  Assume that  $k\geq 2\dim X$ then by the first part $X$ is not $j$-twd
  for some $j>\dim X$.
  Then we apply
  Theorem~\ref{th:ktwd} recursively to conclude.
  We derive identifiability by Proposition~\ref{pro:non_twd_implica_ident}.
\end{proof}

\begin{rmk}
  The first part of Corollary~\ref{cor:twd_defect} extends the bounds
  in \cite{BBC} to non $1$-twd varieties. The main novelty is
  the second part that allows to derive identifiability from non
  defectivity for large enough secant varieties.
\end{rmk}





\section{Application to tensor and structured tensor spaces}\label{sec:application}

As we already mentioned identifiability is particularly interesting for
tensor spaces. In this section we use our main result to explicitly
state identifiability of a variety of tensor spaces.
For this we will consider Segre, Segre-Veronese and Grassmannian varieties
and their $h$-twd properties.

We start with some notation
\begin{notaz}
  The variety $\Sigma(d_{1},...,d_{r};n_{1},...,n_{r})$ is the Segre-Veronese
embedding of $\mathbb{P}^{n_{1}}\times...\times\mathbb{P}^{n_{r}}$
in $\mathbb{P}^{{\prod}\binom{n_{i}+d_{i}}{n_{i}}-1}$ via
the complete linear system $|\oo(d_1,\ldots,d_r)|$.

When all $d_i$'s are one we have the Segre embedding and we let
 $X_{n_{1},...,n_{r}}:=\Sigma(1,\ldots,1;n_1,\ldots,n_r)$ and
 $X_n^r:=\Sigma(1,\ldots,1;n,\ldots,n)\cong(\p^n)^r$.
 The expected generic rank is
 $$gr(\Sigma(d_{1},...,d_{r};n_{1},...,n_{r}))= \lceil
 \frac{{\prod}\binom{n_{i}+d_{i}}{n_{i}}}{(\sum n_i)+1}\rceil$$

 Using the notations in \cite{AOP} we define $$s(\Sigma(d_{1},...,d_{r};n_{1},...,n_{r}):= \lfloor
 \frac{{\prod}\binom{n_{i}+d_{i}}{n_{i}}}{(\sum n_i)+1}\rfloor.$$ 
 For simplicity in the case $n_{1}=...=n_{r}=n$ and
 $d_{1}=...=d_{r}=1$ we
 set
 $$s_{n}^r:=s(\Sigma(d_{1},...,d_{r};n_{1},...,n_{r}))$$
 
 The variety $\mathbb{G}(k,n)$ is the Grassmannian parameterizing $k-$planes in $\mathbb{P}^{n}$ embedded in $\mathbb{P}(\stackrel{k+1}{\bigwedge}V)$
 via the Pl\"ucker embedding.
 The expected generic rank is 
 $$gr(\mathbb{G}(k,n))=\lceil\frac{\tbinom{n+1}{k+1}}{(n-k)(k+1)+1}\rceil$$
 
\end{notaz}

\begin{rmk}
 Note that we always have
 
 $$s(\Sigma(d_{1},...,d_{r};n_{1},...,n_{r})\geq
 gr(\Sigma(d_{1},...,d_{r};n_{1},...,n_{r}))-1$$
 and equality occurs only when $
 \frac{{\prod}\binom{n_{i}+d_{i}}{n_{i}}}{(\sum n_i)+1}$ is not an
 integer. In particular for any $h<s(X)$ we have $\Sec_h(X)\subsetneq\p^N$.
\end{rmk}

The defectivity of Segre and Segre-Veronese varieties is in general
very far from being completely understood,
 \cite{AOP} \cite{AB} \cite{AMR}, but it is in better shape than their
 identifiability. For the latter the best asymptotic
 bounds we are aware of is in \cite{BBC}.

 We start proving the theorem in the introduction.
 \begin{thm}
\label{thm:segre_p1}
Let $X=X_1^k\cong(\p^1)^k$. Then $X$ is not $h$-twd and hence $h$-identifiable in the
following range:
\begin{itemize}
\item [- $(k,h)$]$=(2,1),(3,2),(4,2),(5,4),(6,9)$,
\item[- $k\geq 7$] $h<s(X)$
  \end{itemize}
\end{thm}
\begin{proof} For $k\leq 5$ this is well known, and can be easily
  checked also via a direct computation with commutative algebra
  software. For $k=6$ this has been checked in \cite{BC} by a
  computer aided computation.
  Let us fix $k\geq 7$. By \cite[Theorem 4.1]{CGG} $X$ is never
  defective. In particular the morphism $\pi_h^X$ is generically
  finite for $h\leq s_1^k$.
When  $k\geq 7$ we have
$$2\dim X=2k< \frac{2^k}{k+1}-1< s_1^k,$$
then we can apply Corollary~\ref{cor:twd_defect}. 
\end{proof}
\begin{rmk}
The Theorem answers positively Conjecture 1.2 in \cite{BC} when the
generic rank is an integer, that is
$\frac{2^k}{k+1}\in\N $. For $k\leq 6$ the one listed are the only
identifiable cases.
\end{rmk}

For 3-factors Segre we plug  \cite{CO1} directly in
Theorem~\ref{th:ktwd} to get the following.

\begin{thm}
Let $X=X_{n}^{3}$. Then $X$ is $h$-identifiable  for $h< s(X)$.
\end{thm}

\begin{proof} For $n\leq 7$ the statement is proved in \cite[Theorem
  1.2]{CO1}. For $n>7$, by \cite{L}, the variety  $X$ is not $h$-defective for $h\leq
  s_n^3$ and by the results in \cite{CO1} $X$ is not $h$-twd for
  $h=3n$, confront the table in  \cite[Theorem 1.2]{CO1} . 
  Then we are in the condition to apply Theorem~\ref{th:ktwd}
  recursively to prove that $X$ is not $h$-twd, and hence identifiable,
  for $h<s_n^3$.
\end{proof}

For general diagonal Segre we have a similar statement using \cite{AOP}. 

\begin{thm}
\label{segre ident}
Let $X=X_n^k$, with $n\geq 2$ and $k\geq 4$.
Let
$$n\geq \delta(X)\equiv s_n^k\ mod(n+1)$$
Then $X$ is not $h$-twd and hence $h-$identifiable
for $h< s(X)-\delta(X)$.  In particular when $\delta(X)=0$ $X$
is $h$-identifiable for all $h<s(X)$.
\end{thm}

\begin{proof} 
Using the notations in \cite[Theorem 6.7]{CO1} let $\alpha$ be the
greatest 
integer such that $n+1\geq2^{\alpha}$. 
First we prove the statement for all but finitely many cases.
\begin{claim}
If \[
(k,n)\not\in\left\{
\begin{tabular}{llc}
(k,6) \text{with $k\leq 6$},&
(k,5)  \text{with $k\leq 5$},&\\
(k,4) \text{with $k\leq 5$},&
(k,3) \text{with $k\leq 4$}
\end{tabular} \right\}
\] then $X$ is $h-$identifiable for $h< s(X)-\delta(X)$
\end{claim}
\begin{proof}
By \cite[Theorem 5.2]{AOP} we know that $X$ is not $h$-defective as
long as $h\leq s(X)-\delta(X)$.
The variety  $X_{n}^{k}$ is not $h-$twd
for $$h\leq2^{(k-1)\alpha-(k-1)}=2^{(k-1)(\alpha-1)}$$
by \cite[Theorem 6.7]{CO1}.
Let us assume that $n\neq 2$.
A short hand computation shows that $$2^{(k-1)(\alpha-1)}>dim(X)=kn$$
is satisfied  for every $(k,n)$ in the list. Then, using recursively Theorem~\ref{th:ktwd}, we conclude.
\end{proof}
For the case $n=2$ it is easy to check that the inequality $$s_{2}^{k}=\lfloor\frac{3^{k}}{2k+1}\rfloor-\delta(X_{2}^{k})>4k=2dim(X_{2}^{k})$$ is satisfied for every $k\geq 5$ 
and so we can conclude using Corollary \ref{cor:twd_defect}.
When $(k,n)=(6,6),(5,6)$ we have the inequalities $$s_6^6-\delta(X^6_6)>2\cdot
36=2\dim X^6_6$$
and $$s_6^5-\delta(X^5_6)>2\cdot 30=2\dim X^5_6.$$
Then we conclude by Corollary \ref{cor:twd_defect}.

For all the remaining cases we have that $(n+1)^{k}\leq 15000$ and we
may use the computation in \cite[Theorem 1.1]{COV1} 
to conclude the required identifiability.
\end{proof}




The next class of Segre varieties we treat in details is given
by  $$X[k,n]:=\mathbb{P}^{k}\times(\mathbb{P}^{n})^{k+1}.$$
For these varieties we have
$$gr(X[k,n])=\frac{(k+1)(n+1)^{k+1}}{(k+1)n+k+1}=(n+1)^{k}.$$
In particular $gr(X[k,n])=s(X[k,n])$ is always an integer, that is $X[k,n]$ is
always perfect. Thanks to this special condition we have the
following.

\begin{thm}
Let $X=X[k,n]$ with $n$ odd and  $k>1$. Then $X$ is $h-$identifiable for $h< gr(X)$.
\end{thm}

\begin{proof}
  The proof is entirely similar to that of Theorem~\ref{segre
    ident}. Indeed by \cite[Theorem 5.11]{AOP} we know that all these
  Segre are non defective. 
  If $$(k,n)\neq (4,1),(3,1),(2,1),(2,3),(2,5)$$ the inequality $$(n+1)^{k}>2(k+kn+n)=2dim(X)$$ is satisfied and we conclude 
  using Corollary \ref{cor:twd_defect}.
  For all the exceptional cases we have
  $$(k+1)(n+1)^{k+1}\leq 15000$$
  hence we may apply \cite[Theorem 1.1]{COV1}.
 \end{proof}

\begin{rmk} Defective Segre are expected to be quite rare, beside the
  unbalanced ones, see the conjecture in
  \cite{AOP}. This conjecture has been checked via a computer in many
  cases, \cite{Va} \cite{COV1}. For all these special values our
  argument gives identifiability confirming  the numerical computation
  in \cite{COV1}.
\end{rmk}

Next we apply the same strategy to Segre--Veronese varieties. For this class of
varieties the defectivity results are much weaker and so are
our bounds. Again the special case of binary forms is in better shape.
We start recalling the notation of \cite{LP}.
\begin{dfn}
We say that $(d_{1},...,d_{r};n)$
is special if $$(d_{1},...,d_{r};n)=(2,2a;2a+1),(1,1,2a;2a+1),(2,2,2;7),(1,1,1,1;3)$$
for $a\geq1$. Otherwise $(d_{1},...,d_{r};n)$ is called not
special.
\end{dfn}

\begin{thm}
Let $X=\Sigma(d_{1},...,d_{r};1,..,1)$ with $r=dim X$. Assume that
$(d_{1},...,d_{r};n)$ is not special and $r\geq6$.
Then $X$ is $h-$identifiable
for $h<s(X)$.
\end{thm}

\begin{proof}
We are assuming that  $(d_{1},...,d_{r};n)$ is not special. Then,  by
\cite[Theorem 2.1]{LP}, the variety $X$ is not $h$-defective for
$h\leq gr(X)$. Thanks to Theorem~\ref{thm:segre_p1} we may assume,
without loss of generality that $d_1>1$ and  we have
$$s(X)=\lfloor\frac{(d_{1}+1)\cdots(d_{r}+1)}{r+1}\rfloor\geq\frac{3\cdot2^{r-1}}{r+1}-1.$$
In particular $$\frac{3\cdot2^{r-1}}{r+1}-1>2r=2dimX$$ holds for every
$r\geq6$.

The variety $X$ is not $1$-twd and so we conclude by Corollary \ref{cor:twd_defect}.
\end{proof}

For general Segre-Veronese we have the following.

\begin{thm}
Let $X:=\Sigma(d_{1},...,d_{r},n_{1},...,n_{r})$ be the Segre--Veronese variety. Assume $r\geq 2$,
$$n_{1}^{\lfloor log_{2}(d-1)\rfloor}\geq
2(n_{1}+\ldots +n_{r}),$$
and set $d=d_{1}+\ldots +d_{r}$.
Then $X$ is  $h-$identifiable
for $h\leq n_{1}^{\lfloor log_{2}(d-1)\rfloor}-1$.
\end{thm}

\begin{proof} By  \cite[Theorem 1.1]{AMR} $X$ is not $h$-defective for
  $$h\leq  n_{1}^{\lfloor
    log_{2}(d-1)\rfloor}-(n_{1}+...+n_{r})+1. $$
  In our numerical assumptions $\Sec_h(X)\subsetneq\p^N$ and we may
  assume
  $$ h\geq 2\dim X.$$
  Then we conclude by Corollary~\ref{cor:twd_defect}.
\end{proof}

\begin{rmk}
  For the Veronese variety of $\p^n$, that is $\Sigma(d_1,n_1)$ it is
  easy, via Corollary~\ref{cor:twd_defect} and \cite{AH}, to reprove the identifiability
  results in \cite{Me} and \cite{COV2}. 
\end{rmk}

As in the Segre case, for special classes of Segre--Veronese there are
better non defectivity results.
Here we recall  the notation in \cite{AB}.
Let $X:=\sum(1,2;m,n)$ be the Segre-Veronese variety $\mathbb{P}^{m}\times\mathbb{P}^{n}$
embedded by $\mathcal{O}(1,2)$ in $\mathbb{P}^{N}$ where $$N=(m+1)\tbinom{n+2}{2}-1$$

Let \[
r(m,n)=
\begin{cases}
m^3-2m & \text{if $m$ even and $n$ odd} \\
\frac{(m-2)(m+1)^2}{2} & \text{otherwise}
\end{cases}
\]
and
$$s(X)=\lfloor\frac{(m+1)\tbinom{n+2}{2}}{m+n+1}\rfloor$$
the meaningful numbers of $X$. With this in mind we have the following.

\begin{cor}
  Let $X=\sum(1,2;m,n)$. If $n>r(m,n)$ and
  $$\lfloor\frac{(m+1)\tbinom{n+2}{2}}{m+n+1}\rfloor\geq 2(m+n)$$
then $X$ is not $h-$twd and hence $h-$identifiable for $h<s(X)$.
\end{cor}

\begin{proof} In our range $X$ is not $h$-defective by \cite[Theorem
  1.1]{AB} and $\Sec_h(X)\subsetneq\p^N$.
  Moreover
  $$s(X)=\lfloor\frac{(m+1)\tbinom{n+2}{2}}{m+n+1}\rfloor\geq
  2(m+n)=2\dim X$$
and we may apply Corollary~\ref{cor:twd_defect} to conclude.
\end{proof}

Let us consider now the case of $\mathbb{P}^{m}\times\mathbb{P}^{n}$
embedded with $\mathcal{O}(1,d)$ for $d\geq3$.
\begin{cor}
Let $X=\sum(1,d;m,n)$ with $d\geq3$ and $m,n\geq1$. Let $$s(X)=max\left\{ s\in\mathbb{N}|\text{s is a multiple of $(m+1)$ and $s$}\leq\lfloor\frac{(m+1)\binom{n+d}{d}}{m+n+1}\rfloor\right\}$$
If $s(X)>2(m+n)$ then $X$ is not $h$-twd and hence $h$-identifiable
for $h<s(X)$.
\end{cor}

\begin{proof}
By \cite[Theorem 2.3]{BCC} $X$ is not $h$-defective for $h\leq s(X)$
and $\mathbb{S}ec_{h}(X)\subsetneq\mathbb{P}^{(m+1)\binom{n+d}{d}-1}$.

$X$ is smooth, in particular it is not $1$-twd.
Since $$s(X)>2(m+n)=2dim(X)$$
we can apply Corollary \ref{cor:twd_defect} to conclude.
\end{proof}

\begin{rmk}
Similar statements about subgeneric identifiability of $\p^{n} \times \p^{1}$ embedded with $\mathcal{O}(a,b)$ can be derived applying Corollary \ref{cor:twd_defect}
using the non defectivity results in \cite{BBC1}.
\end{rmk}

Finally we consider Grassmannian varieties.  For this class of tensor
spaces very few is known about identifiability. To the best of our
knowledge the following is the first non computer aided result for them.

\begin{thm}
Let $X=\mathbb{G}(k,n)$ such that $2k+1\leq n$. Assume that 
$$\lfloor \left(\frac{n+1}{k+1}\right)^{\lfloor log_{2}(k)\rfloor}\rfloor\geq2(n-k)(k+1)$$.
Then $X$ is $h-$identifiable for 
$$h\leq\left(\frac{n+1}{k+1}\right)^{\lfloor log_{2}(k)\rfloor}-1$$
\end{thm}

\begin{proof}
By \cite[Theorem 5.4]{MR} in our numerical range $X$ is not
$h-$defective and $\Sec_h(X)\subsetneq \p^N$. Then we conclude by
Corollary~\ref{cor:twd_defect}.
\end{proof}

The technique we developed can be applied to many other classes of
varieties, once it is known their defectivity behavior. 
As a sample we conclude with  the following example.
\begin{es}
C. Am\'endola, J.-C. Faugère, K. Ranestad and B. Sturmfels in \cite{AFS} and \cite{ARS} studied the Gaussian moment variety
$$\mathcal{G}_{1,d}\subset\mathbb{P}^{d}$$
whose points are the vectors of all moments of degree $\leq d$ of a $1-$dimensional Gaussian
distribution. 
They proved that $\mathcal{G}_{1,d}$ is a surface for every
$d$  and $\mathbb{S}ec_{h}(\mathcal{G}_{1,d})$
has always the expected dimension.
In \cite[Example 5.8]{BBC} it is shown that $\mathcal{G}_{1,d}$ is not
uniruled by lines, in particular it is not $1$-twd. As usual let
$$s(\mathcal{G}_{1,d})=\lfloor\frac{d+1}3\rfloor\geq gr(\mathcal{G}_{1,d})-1 $$
Then by Corollary~\ref{cor:twd_defect} $\mathcal{G}_{1,d}$ is
$h$-identifiable, for $h<s(\mathcal{G}_{1,d})$ when $d\geq 14$.
\end{es}

\end{document}